\documentclass[11pt,a4paper,twoside]{amsart}

\usepackage{amsmath,amssymb,amsthm}
\usepackage{inputenc}
\usepackage{tikz} 
\usepackage{float}
\floatstyle{boxed} 
\usepackage{pdfsync} 
\usepackage{graphicx,subfigure}
\usepackage{caption}
\usepackage[T1]{fontenc} 
\usepackage{enumitem}
\usepackage{mathtools}
\usepackage{mathrsfs}
\usepackage{xcolor}
\usepackage[colorlinks=true, linkcolor=blue, citecolor=red, urlcolor=magenta,hyperfootnotes=false]{hyperref}
\usepackage{cleveref}
\usepackage{etoolbox}
\usepackage{upgreek}

\def\@begintheorem#1#2{\par\bgroup{\sc #1 \ #2. }  \it \\\ignorespace }
\def\@opargbegintheorem#1#2#3{\par\bgroup{\sc #1\ #2 \ (#3).}  \it  \ignorespace}
\def\@endtheorem{\egroup}

\theoremstyle{plain}
\newtheorem{theorem}{Theorem}[section]

\theoremstyle{definition}
\newtheorem{definition}[theorem]{Definition}

\theoremstyle{remark}
\newtheorem{remark}[theorem]{Remark}

\theoremstyle{plain}

\theoremstyle{plain}
\newtheorem{lemma}[theorem]{Lemma}

\theoremstyle{plain}

\theoremstyle{plain}
\newtheorem{proposition}[theorem]{Proposition}

\numberwithin{equation}{section}

\newcommand{\SL}{\S\cdot L}

\newcommand{\C}{{\mathbb C}}

\newcommand{\D}{\mathcal{D}}
\newcommand{\Z}{{\mathbb Z}}

\newcommand{\R}{{\mathbb R}}

\newcommand{\I}{\mathbb{I}}

\newcommand{\V}{\mathbb{V}}

\renewcommand{\S}{{\mathbf{S}}}

\newcommand{\Rt}{{\R}^3}

\newcommand{\hx}{\hat{x}}
\renewcommand{\H}{\mathcal{H}}

\DeclareMathOperator{\Realpart}{Re}
\renewcommand{\Re}{\Realpart}

{\left\lbrace\begin{array}{@{}l@{}}}%
{\end{array}\right.}

\DeclarePairedDelimiter{\abs}{\lvert}{\rvert}
\DeclarePairedDelimiter{\norm}{\lVert}{\rVert}
\DeclarePairedDelimiter{\seq}{\lbrace}{\rbrace}

\newcommand{\ignora}[1]{}

\title[Boundary triples for the Dirac-Coulomb operator]%
{Boundary triples for the Dirac operator with
  Coulomb-type spherically symmetric perturbations}
\date{\today}

\author[B.~Cassano]{Biagio Cassano}
\address{B.~Cassano, Department of Theoretical Physics, NPI, Academy of Sciences, 25068 \v{R}e\v{z}, Czechia)}
\email{cassano@ujf.cas.cz}

\author[F.~Pizzichillo]{Fabio Pizzichillo}
\address{F.~Pizzichillo, CNRS \& CEREMADE, Universit\'e Paris-Dauphine, PSL Research University, F-75016 Paris, France}
\email{pizzichillo@ceremade.dauphine.fr}

\usepackage{csquotes}

\begin{document}
\begin{abstract}
  We determine explicitly a boundary triple for 
 the Dirac operator $H:=-i\alpha\cdot \nabla + m\beta + \mathbb V(x)$ in
 $\mathbb R^3$, for $m\in\mathbb R$ and $\mathbb V(x)= |x|^{-1} ( \nu
 \mathbb{I}_4 +\mu \beta -i  \lambda \alpha\cdot{x}/|x|\,\beta)$, with
 $\nu,\mu,\lambda \in \mathbb R$.
 Consequently we determine all the self-adjoint realizations of $H$
 in terms of the behaviour of the functions of their domain in the
 origin.
 When $\sup_{x} |x| |\mathbb V(x)| \leq 1$,
 we discuss the problem of selecting the \emph{distinguished}
 extension requiring
 that its domain is included in the domain of the appropriate quadratic form.
 
\end{abstract}
\maketitle

\section{Introduction and main results}
In this paper we determine a boundary triple and describe all the self-adjoint realizations of the 
differential operator 
\begin{equation}\label{eq:def.H}
H:=H_0+\V
\end{equation}
where $H_0$ is the free Dirac operator in $\R^3$ defined by 
\begin{equation}\label{eq:defn.H0}
H_0:=-i\alpha\cdot\nabla+m\beta, 
\end{equation}
with $m\in\R$,
\begin{equation*}
  \beta:=
  \begin{pmatrix}
\mathbb{I}_2&0\\
0&-\mathbb{I}_2
\end{pmatrix},
\quad
\alpha:=(\alpha_1,\alpha_2,\alpha_3),
\quad
\alpha_j:=\begin{pmatrix}
0& {\sigma}_j\\
{\sigma}_j&0
\end{pmatrix}\quad \text{for}\ j=1,2,3,
\end{equation*}
and 
$\sigma_j$ are the \emph{Pauli matrices}
\[
\quad{\sigma}_1 =
\begin{pmatrix}
0 & 1\\
1 & 0
\end{pmatrix},\quad {\sigma}_2=
\begin{pmatrix}
0 & -i\\
i & 0
\end{pmatrix},
\quad{\sigma}_3=
\begin{pmatrix}
1 & 0\\
0 & -1
\end{pmatrix},
\]
and finally
\begin{equation}\label{eq:def.V}
\V(x):= \frac{1}{|x|}\left( \nu \mathbb{I}_4 +\mu \beta +
  \lambda \left(-i\alpha\cdot\frac{x}{\abs{x}}\,\beta \right)\right), \quad  \text{for } x\neq 0,  
\end{equation}
where $\nu$, $\lambda$ and $\mu$ are real numbers, 
and $\mathbb{I}_4$ is the 
$4 \times 4$ identity matrix. 

The operator $H_0+\V$ describes the motion of
relativistic $\frac12$--spin
particles in the external potential $\V$. In detail, setting
\begin{equation*}
  \mathbb{V} = 
  \mathbb{V}_{el} + 
  \mathbb{V}_{sc} + 
  \mathbb{V}_{am} :=
  v_{el}(x) \mathbb{I}_4 +
  v_{sc}(x) \beta +
  v_{am}(x) \left(-i \alpha\cdot\frac{x}{\abs{x}}\,\beta \right),
\end{equation*}
for real valued $v_{el},v_{sc}, v_{am}$, the potentials $\mathbb{V}_{el}, \mathbb{V}_{sc},  \mathbb{V}_{am}$ are 
called respectively
 \emph{electric}, \emph{scalar}, and \emph{anomalous magnetic} potential. 
This particular class of potentials has the property that, in the
case that
$v_{el},v_{sc}, v_{am}$ only depend on the radial variable, 
the action of $H_0+\V$ leaves invariant the \emph{partial
  wave subspaces}
 (see below). Moreover, in the case that they have a
singularity $\sim \abs{x}^{-1}$ in the origin the potential has the same scaling as the Dirac Operator.

The dynamics of quantum systems is described in terms of
self-adjoint operators, as shown by the Stone's theorem, see e.g.~\cite{reedsimon1}. 
For this reason, it is a primary task to
describe all the self-adjoint extensions (if any exists) of a given symmetric
operator associated with a physical system. 
Von Neumann gave the first complete solution to this problem:
his theory is fully general and completely describes all the
self-adjoint extensions of every densely defined and symmetric 
operator in an abstract Hilbert space in terms of unitary operators
between its deficiency spaces, see e.g.~\cite{reedsimon2}.
Von Neumann's theory works at an abstract level: for specific classes of operators, it is desirable to have a more concrete
characterization of the self-adjoint extensions. 
In many cases, self-adjoint operators arise
when one introduces some boundary conditions for a differential
expression: perturbing operators with potentials with a singularity in one point, one would like to establish a direct link between 
self-adjoint extensions and behaviour in the point of the functions in
their domain.
Referring to \cite{bruning2008spectra, facchi2018self} for a general overview
on the theories of self-adjoint extensions,  we cite here the  
theory of \emph{boundary triples}, see \cite{schmudgen2012unbounded,
  bruning2008spectra, de2008intermediate, pavlov1987theory} and references therein, that
gives this desired description. The main result of this paper (\Cref{thm:boundary})
is the explicit determination of a boundary triple for the operator
$H$: thanks to this, we are then able to describe all the self-adjoint
realizations in terms of the behaviour in the origin of the functions in the
domain.

A vast literature has been dedicated to the problem of  the self-adjointness of
perturbed Dirac operators.
Remanding to the introduction of \cite{coulombbiagio}, to the 
survey \cite{gallone2017self}
and to the book \cite{thaller} for more details, we list here some relevant works. 
In \cite{kato1951fundamental} it was observed that
thanks to the Hardy inequality
\begin{equation}\label{eq:classicalhardy}
\frac{1}{4}\int_{\Rt}\frac{|f|^2}{|x|^2}\, dx\leq \int_{\Rt} |\nabla f|^2\,dx, \quad\text{for}
\
f\in C^\infty_c(\Rt),
\end{equation}
and the Kato-Rellich Theorem it is possible to prove that,
for $|\nu|\in\left[0,\frac12\right)$, the operator $H_0+\nu/\abs{x}$
is essentially self-adjoint on $C^\infty_c(\Rt)^4$ and self-adjoint on
$\D(H_0)=H^1(\Rt)^4$.
In fact the optimal range for the self-adjointness is  $|\nu|\in
\left[0,\frac{\sqrt{3}}{2}\right)$, as shown in
\cite{  gustafson1973some, rellich1953eigenwerttheorie, schmincke1972essential, weidmann1971oszillationsmethoden}.
For $|\nu|>\sqrt{3}/2$,  $H_0+\nu/\abs{x}$ is not essentially self-adjoint and infinite self-adjoint extensions can be constructed. 
Among these, for $|\nu|\in\left(\frac{\sqrt{3}}{2},1\right)$ there
exists one \emph{distinguished} extension $H_S$ such that
\begin{equation}\label{eq:schminke}
\D(H_D) \subset \D(r^{-1/2})^4=\seq{\psi\in L^2(\Rt)^4:|x|^{-1/2}\psi\in
  L^2(\Rt)^4}
\end{equation}
or equivalently $\D(H_D) \subset H^{1/2}(\Rt)^4 $: in other words, 
one requires that all the functions in the domain of the extension are in the form domain of the potential and the momentum. For details see 
\cite{burnap1981dirac, gallonemichelangeli2017self, klaus1979characterization, nenciu1976self, schmincke1972distinguished, wust1975distinguished}.
For  $|\nu|\geq 1$ many self-adjoint extensions can be built, and for  $|\nu|> 1$ none
appears to be \emph{distinguished} in some suitable sense, see
\cite{hogreve2012overcritical, voronov2007dirac, xia1999contribution}.
The definition of a distinguished extension for the case $\abs{\nu}=1$ 
has been settled in \cite{estebanloss}, where it is considered a pontential $V:\Rt\to \R$ such that that for some constant 
$c(V)\in (-1,1)$, $\Gamma:=\sup(V)<1+c(V)$ and for every $\varphi\in C^\infty_c(\Rt)^2$,
\begin{equation}\label{eq:diseq:EL}
\int_{\Rt}\left( \frac{|\sigma\cdot\nabla \varphi|^2}{1+c(V)-V}+\left(1+c(V)+V\right)|\varphi|^2\right) dx\geq 0.
\end{equation}   
In particular, for an electrostatic potential
$\V(x):= V(x) \I_4$,
$-\nu \abs{x}^{-1} \leq V(x) < 1 + \sqrt{1-\nu^2}$, $0<\nu\leq1$, 
the operator $H_0+\V$ is self-adjoint on a suitable domain.
If $0<\nu<1$, the self-adjoint extension described is the
distinguished one, as also shown in \cite{muller2016minimax};  for $\nu=1$, the self-adjoint extension described is the distinguished one, since continuous prolongation of the sub-critical case can cover it. 
Recently, in \cite{els2017domains}, it is shown that this extension can be obtained  as the limit in the norm resolvent sense of
potentials where the singularity has been removed with a cut-off
around the singularity.

The approach of \cite{kato1951fundamental} could be used independently on the
spherical symmetry of the potential: $H_0 + \V$ is self-adjoint when
$\V$ is  a $4\times 4$ Hermitian real-valued matrix potential $\V$ such that
\begin{equation*}
|\V(x)|\leq a\frac{1}{|x|}+b, \quad
x \in \R^{3}\setminus\{0\},
\end{equation*}
with $b\in\R$ and $a<1/2$, see \cite[Theorem V 5.10]{kato2013perturbation}.
In \cite{arrizabalaga2011distinguished, adv2013self,kato1983holomorphic} 
more general $4\times 4$ matrix-valued measured functions 
$\V$ are considered, in the assumption that $\abs{x}|\V(x)|\leq
\nu < 1$, and a distinguished self-adjoint extension (in the sense of
\eqref{eq:schminke}) is constructed, exploiting the 
\emph{Kato-Nenciu} inequality
\begin{equation}\label{eq:kato.nenciu}
\int_{\Rt}\frac{|\psi|^2}{|x|}\,dx \leq \int_{\Rt}\abs*{(-i\alpha\cdot\nabla+m\beta+ i\epsilon)\psi}^2|x|\,dx, \quad\text{for}
\ \psi\in C^\infty_c(\Rt)^4,  m, \epsilon \in \R.
\end{equation}

In our previous work \cite{coulombbiagio}, we considered 
matrix-valued potentials as in \eqref{eq:def.V} and 
we investigated the
existence of self-adjoint extensions $T$ such that
\begin{equation}\label{eq:T}
 \mathring H_{min} \subseteq T = T^* \subseteq H_{max},
\end{equation}
where the \emph{minimal operator} $\mathring H_{min}$
and the \emph{maximal operator} $H_{max}$ are defined
 as follows:
 \begin{align}
   \label{eq:defn.minimal.operator} 
    &\mathcal{D}(\mathring{H}_{min}):=C^\infty_c(\Rt\setminus \{0\})^4, 
    & & 
    \mathring H_{min}\psi :=H \psi \quad \text{for }\psi \in \mathcal{D}(\mathring H_{min}),
   \\
   \label{eq:defn.maximal.operator} 
   &\mathcal{D}(H_{max}):=\seq{\psi\in L^2(\Rt)^4: H\psi\in L^2(\Rt)^4}, 
   & & 
       H_{max} \psi :=H \psi \quad \text{for }\psi \in \mathcal{D}(H_{max}),
\end{align}
where $H \psi$ in \eqref{eq:defn.minimal.operator} is computed 
in the classical sense and in \eqref{eq:defn.maximal.operator} $H \psi\in L^2(\Rt)^4$ has to be read in the
distributional sense.
It is easy to see that $\mathring H_{min}$ is symmetric and $(\mathring H_{min})^*=H_{max}$.
The strategy of \cite{coulombbiagio} consists in considering the
self-adjointness of $H_0 + \V$ on the \emph{partial
  wave subspaces}:
such spaces are left invariant by $H_0$ and potentials $\mathbb V$ as in \eqref{eq:def.V}.
We sketch here this topic, referring to \cite{coulombbiagio} and \cite[Section 4.6]{thaller} for further details.

Let $Y^l_n$ be the spherical harmonics. They are defined for $n = 0, 1, 2, \dots$, and $l =-n,-n + 1,\dots , n,$ and they satisfy $\Delta_{\mathbb{S}^2} Y^l_n= n(n + 1)Y^l_n$, where $\Delta_{\mathbb{S}^2}$ denotes the usual spherical Laplacian. Moreover, $Y^l_n$ form a complete orthonormal set in $L^2(\mathbb{S}^2)$.
For $j = 1/2, 3/2, 5/2, \dots , $ and $m_j = -j,-j + 1, \dots , j$, set
\begin{align*}
\psi^{m_j}_{j-1/2}&:=
\frac{1}{\sqrt{2j}}
\left(\begin{array}{c}
\sqrt{j+m_j}\,Y^{m_j-1/2}_{j-1/2}\\
\sqrt{j-m_j}\,Y^{m_j+1/2}_{j-1/2}\\
\end{array}\right),
\\
\psi^{m_j}_{j+1/2}&:=\frac{1}{\sqrt{2j+2}}
\left(\begin{array}{c}
\sqrt{j+1-m_j}\,Y^{m_j-1/2}_{j+1/2}\\
-\sqrt{j+1+m_j}\,Y^{m_j+1/2}_{j+1/2}\\
\end{array}\right);
\end{align*}
then  $\psi^{m_j}_{j\pm1/2}$ form a complete orthonormal set in $L^2(\mathbb{S}^2)^2$.
Moreover, we set 
\begin{equation*}
r=|x|,\quad\hat x = x / |x|\quad\text{and}\quad
L=-ix\times\nabla\quad\text{for }x \in \Rt\setminus\{0\}.
\end{equation*}
Then
\begin{equation*}
(\sigma\cdot\hx)\psi^{m_j}_{j\pm 1/2}=\psi^{m_j}_{j\mp1/2},
\quad\text{and}\quad
(1+\sigma\cdot L)\psi^{m_j}_{j\pm 1/2}=\pm(j+1/2)\psi^{m_j}_{j\pm 1/2},
\end{equation*}
where $\sigma=(\sigma_1,\sigma_2,\sigma_3)$ is the vector of \textit{Pauli's matrices}.
For $k_j:=\pm(j+1/2)$ we set
\begin{equation*}
\Phi^+_{m_j,\pm(j+1/2)}:=
\left(\begin{array}{c}
i\,\psi^{m_j}_{j\pm1/2}\\
0
\end{array}\right),
\quad
\Phi^-_{m_j,\pm(j+1/2)}:=
\left(\begin{array}{c}
0\\
\psi^{m_j}_{j\mp1/2}
\end{array}\right).
\end{equation*}
Then, the set $\{\Phi^+_{m_j,k_j},\Phi^-_{m_j,k_j}\}_{j,k_j,m_j}$ is a 
complete orthonormal basis of $L^2(\mathbb{S}^2)^4$.
We prescribe the following ordering for the triples
$(j,m_j,k_j)$, for $j=\frac12,\frac32,\dots ; \, m_j = -j,\dots,j; \,
k_j=j+1/2,-j-1/2$:
\begin{equation}\label{eq:ordering}
    \begin{split}
   & \left(\frac12,-\frac12,1\right), \left(\frac12,\frac12,1\right),
       \left(\frac12,-\frac12,-1\right),
       \left(\frac12,\frac12,-1\right),
       \\
     & \left(\frac32,-\frac32,2\right), \left(\frac32,-\frac12,2\right),
     \left(\frac32,\frac12,2\right),\left(\frac32,\frac32,2\right),
     \\
     &\left(\frac32,-\frac32,-2\right), \left(\frac32,-\frac12,-2\right),
     \left(\frac32,\frac12,-2\right),\left(\frac32,\frac32,-2\right),
     \dots, \\
     & \left(j, -j, j + \frac12\right),\dots, \left(j,j, j+\frac12\right),
     \left(j, -j, -j - \frac12\right),\dots, \left(j,j, -j-\frac12\right), \dots
  \end{split} 
\end{equation}

We define the following space:
\begin{equation*}
\mathcal{H}_{m_j,k_j}:=
\seq*{ 
\frac{1}{r}\left(
f^+_{m_j,k_j}(r)\Phi^+_{m_j,k_j}(\hx)+
f^-_{m_j,k_j}(r)\Phi^-_{m_j,k_j}(\hx)\right)
\in L^2(\Rt)
\mid
f^\pm_{m_j,k_j}\in L^2(0,+\infty)}.
\end{equation*}
From \cite[Theorem 4.14]{thaller} we know that the operators $\mathring H_{min}$ and $H_{max}$ leave the partial wave subspace $\H_{m_j,k_j}$ invariant and their action can be decomposed in terms of the basis $\seq{\Phi^+_{m_j,k_j},\Phi^-_{m_j,k_j}}$ as follows:
\begin{align}
\label{eq:H.min}
\mathring H_{min}&\cong
\bigoplus_{j=\frac{1}{2},\frac{3}{2},\dots}^\infty\,\bigoplus_{m_j=-j}^j\,
\bigoplus_{k_j=\pm(j+1/2)}\,
h_{m_j,k_j},
\\
\notag
 H_{max}&\cong
\bigoplus_{j=\frac{1}{2},\frac{3}{2},\dots}^\infty\,\bigoplus_{m_j=-j}^j\,
\bigoplus_{k_j=\pm(j+1/2)}\,
{h}^*_{m_j,k_j},
\end{align}
where ``$\cong$'' means that the operators are unitarily equivalent, with 
\begin{equation}\label{eq:dirac.spherical}
\begin{split}
& {D}(  h_{m_j,k_j})= C^\infty_c(0,+\infty)^2,
\\
& h_{m_j,k_j}
( f^+,  f^-):
=\left(
\begin{array}{cc}
m+\frac{\nu+\mu}{r} & -\partial_r+\frac{k_j+\lambda}{r}\\
\partial_r+\frac{k_j+\lambda}{r} & -m+\frac{\nu-\mu}{r}
\end{array}\right)
\begin{pmatrix}
  f^+ \\
 f^-
\end{pmatrix};
\end{split}
\end{equation}
and 
\begin{equation}\label{eq:dirac.spherical*}
\begin{split}
& {D}( h^*_{m_j,k_j})= \seq{(f^+,f^-)\in L^2(0,+\infty): h^*_{m_j,k_j}(f^+,f^-) \in L^2(0,+\infty)^2},\\
& h^*_{m_j,k_j}
(
  f^+, f^-):
=\left(
\begin{array}{cc}
m+\frac{\nu+\mu}{r} & -\partial_r+\frac{k_j+\lambda}{r}\\
\partial_r+\frac{k_j+\lambda}{r} & -m+\frac{\nu-\mu}{r}
\end{array}\right)
\begin{pmatrix}
  f^+ \\
 f^-
\end{pmatrix};
\end{split}
\end{equation}
where $h^*_{m_j,k_j}(f^+,f^-)$ has to be read in the distributional sense as done in \eqref{eq:defn.maximal.operator}.
It is easy to see that $h^*_{m_j,k_j}$ is the adjoint of $h_{m_j,k_j}$.

The main result of \cite{coulombbiagio} is the classification of all
the self-adjoint extensions $t_{m_j,k_j}$ such that
$h_{m_j,k_j}\subseteq  t_{m_j,k_j}=t_{m_j,k_j}^* \subseteq h^*_{m_j,k_j}$:
as an immediate consequence, we can build up self-adjoint operators
$T$ as in \eqref{eq:T} setting 
\begin{equation*}
    T \cong \bigoplus_{j=\frac{1}{2},\frac{3}{2},\dots}^\infty\,
    \bigoplus_{m_j=-j}^j\, \bigoplus_{k_j=\pm(j+1/2)}\,
    t_{m_j,k_j}.
\end{equation*}

The self-adjointness of $t_{m_j,k_j}$ is related to the quantity
\begin{equation}
  \label{eq:defn.delta}
  \delta_{k_j}=\delta_{k_j}(\lambda, \mu, \nu):=(k_j+\lambda)^2+\mu^2-\nu^2.
\end{equation}
In \cite[Theorems 1.1, 1.2, 1.3]{coulombbiagio} we show that
if $\delta_{k_j} \geq 1/4$ then
$t_{m_j,k_j}$ is essentially self-adjoint 
and if $\delta_{k_j} < 1/4$
that there exists a one (real) parameter family 
$\left(t(\theta)_{m_j,k_j}\right)_{\theta \in [0,\pi)}$
    of self-adjoint extensions such that $h_{m_j,k_j}\subset
    t(\theta)_{m_j,k_j}=t(\theta)_{m_j,k_j}^* \subset h^*_{m_j,k_j}$.
In conclusion, we can define a family of
self-adjoint extensions parametrised by $d$ real parameters, with
\begin{equation}
  \label{eq:defn.d}
  d:= 
\sum_{\substack{j,m_j,k_j \\ (k_j+\lambda)^2+\mu^2-\nu^2<1/4}} 
1
\,
=
\sum_{\substack{k \in \Z\setminus \{0\} \\ (k+\lambda)^2 + \mu^2
      -\nu^2 < 1/4}}
  2 \abs{k}.
\end{equation}
In this paper we show that the totality of the self-adjoint extensions is a much richer
set. Indeed, they are in one-to-one correspondence with the unitary
matrices 
 \begin{equation*}
    \mathcal U (d) := \{U \in \C^{d \times d} : U^* U = U U^* = \I_{d} \},  
  \end{equation*}
 that is they are a family of $d^2$ real parameters. This correspondence relates the self-adjoint extensions with the
 behaviour in the origin of the functions in their domain.
 In order to do so, we
exploit the theory of the boundary triples: we remind here its
definition, following the notations from
\cite[Definition 1.7]{bruning2008spectra}.

\begin{definition}\label{defn:boundary}
  Let $E: \mathcal{D}(E) \subseteq
\mathcal{H} \to \mathcal{H}$ be a closed linear operator in a
Hilbert space $\mathcal{H}$, and let $\mathcal G$ be an other Hilbert space.
Let $\Gamma_1, \Gamma_2 : \mathcal{D}(E) \to \mathcal{G}$ be linear
maps, and finally define  $(\Gamma_1, \Gamma_2) : \mathcal{D}(E) \to \mathcal{G}
\oplus \mathcal{G} $ as $(\Gamma_1,\Gamma_2)\psi:=(\Gamma_1\psi,\Gamma_2\psi)$, for any $\psi \in \mathcal{D}(E)$.
We say that the triple $(\mathcal{G},\Gamma_1,\Gamma_2)$ is a
\emph{boundary triple} for $E$ if and only if:
\begin{gather}
\label{eq:boundarytriple.cond1}
  \langle \psi, E \widetilde{\psi} \rangle_{\mathcal{H}}
-
\langle E \psi, \widetilde{\psi} \rangle_{\mathcal{H}}
=
\langle \Gamma_1 \psi ,\Gamma_2 \widetilde{\psi} \rangle_{\mathcal{G}}
-
\langle \Gamma_1 \psi, \Gamma_2 \widetilde{\psi}
\rangle_{\mathcal{G}}
\quad \text{ for all } \psi, \widetilde{\psi} \in \mathcal{D}(E);
 \\
\label{eq:boundarytriple.cond2}  
\text{ the map }(\Gamma_1, \Gamma_2) : \mathcal{D}(E) \to \mathcal{G}
\oplus \mathcal{G} \text{ is surjective};  
 \\
\label{eq:boundarytriple.cond3}
\text{ the set } \ker(\Gamma_1, \Gamma_2) \text{ is dense in $\mathcal{H}$}.
\end{gather}
\end{definition}
The theory of the boundary triples is well developed and powerful: the
explicit knowledge of a boundary triple for a symmetric and closed operator
can be used to obtain many important results.
In this paper we exploit it to describe all the self-adjoint
extensions: the following proposition is
consequence of Theorem 1.2, Proposition 1.5 and Theorem 1.12
in \cite{bruning2008spectra}, or equivalently of
Proposition 14.4 and Theorem 14.10 in \cite{schmudgen2012unbounded},
hence the proof is omitted.
\begin{proposition}\label{thm:main-abs}
Let $E_0$ be a symmetric operator on a Hilbert space $\H$ and let $(\mathcal{G}, \Gamma_1,\Gamma_2)$ be a boundary triple for $E^*:=(E_0)^*$. Then
the following hold:
\begin{enumerate}[label=({\roman*})]
\item if $\mathcal{G}=\seq{0}$, $E_0$ is essentially self-adjoint;
\item if $\mathcal{G}\neq \seq{0}$,  $E_0$ has many self-adjoint extensions. They can be classified in the following equivalent ways:
  \begin{itemize}
    \item For any $A, B$ bounded linear operators on $\mathcal{G}$, the extension $E_{A,B}$
with domain
\begin{equation}\label{eq:defn.TAB-abs}
  \D(E_{A,B})=
 \seq{
 \psi\in \D(E^*):\ A \Gamma_1(\psi) =
 B \Gamma_2(\psi)}
\end{equation}
is self-adjoint if and only if
\begin{gather}
  \label{eq:conditions.A,B.1-abs}
  AB^* = B A^*,
  \\
  \label{eq:conditions.A,B.2-abs}
  \text{ker}
  \begin{pmatrix}
    A & -B \\
    B & A
  \end{pmatrix}
  =0.
\end{gather}
  \item There exists a one-to-one correspondence 
between the self-adoint extensions of $E_0$
and the unitary operators $\mathcal{U}(\mathcal{G})$.
For $U\in \mathcal{U}(\mathcal{G})$, the corresponding self-adjoint extension $E_U$ has domain
 \begin{equation}\label{eq:descrizione.domini-abs}
 \D(E_U)=
 \seq{
 \psi\in \D(E^*):\ i (\mathbb{I}_{\mathcal{G}} + U  )\Gamma_1(\psi) =
 (\mathbb{I}_{\mathcal{G}} - U) \Gamma_2(\psi)}.
\end{equation}
  \end{itemize}
  \end{enumerate}
\end{proposition}
\begin{remark}
  The descriptions of the self-adjoint extensions in
  \eqref{eq:defn.TAB-abs} and \eqref{eq:descrizione.domini-abs}
  are equivalent and both useful and interesting.
  Indeed,
  \eqref{eq:defn.TAB-abs} is useful for the applications:
  for example we will exploit it in \Cref{rmk:distinguished}
  to determine the distinguished extension for the Dirac-Coulomb
  operator. The description in \eqref{eq:descrizione.domini-abs}
  is interesting from a more theoretical point of view, since
  it gives a one-to-one correspondence between the self-adjoint
  extensions and the elements of the unitary operators on
  $\mathcal{G}$, allowing to label this extensions with
  a unique choice of parameters.
\end{remark}

We introduce some notations.

\begin{definition}\label{def:boun-trip-sferical}
Let 
\begin{equation*}
\psi(x)= \sum_{j,m_j,k_j} \frac{1}{r}\left(
f^+_{m_j,k_j}(r)\Phi^+_{m_j,k_j}(\hx)+
f^-_{m_j,k_j}(r)\Phi^-_{m_j,k_j}(\hx)
\right)
\in \mathcal{D}(H_{max})
\end{equation*}
and set $f_{m_j,k_j}:=\left(f^+_{m_j,k_j},f^-_{m_j,k_j}\right)\in \D(h^*_{m_j,k_j})$.
We select in the order \eqref{eq:ordering} the triples $(j,m_j,k_j)$
 such that 
 $\delta_{k_j}:=(k_j+\lambda)^2 + \mu^2 - \nu^2 < 1/4$
 and we denote this ordered set $I$: we have that $I$ has
 exactly $d$ elements.
 Moreover  we set
 \begin{equation}
   \label{eq:defn.gammakj}
   \gamma_{k_j}:=\sqrt{\abs{\delta_{k_j}}},
   \quad \text{ for all } j=\frac12,\frac32,\dots. 
 \end{equation}
Then, for any $(j,m_j,k_j)\in I$:
\begin{enumerate}[label=\emph{({\roman*})}]
\item \label{item:boundary.sottocritico}
  if $0 < \delta_{k_j} < 1/4$ from \cite[Proposition 3.1, \emph{(iii)}]{coulombbiagio} we know that
\begin{equation}\label{eq:limite.delta.12}
\lim_{r\to 0}
\abs*{
    \begin{pmatrix}
    f_{m_j,k_j}^+(r) \\ f_{m_j,k_j}^-(r)
  \end{pmatrix}
  - D_{k_j} 
  \begin{pmatrix}
    A^+ r^{\gamma_{k_j}} \\ A^- r^{-\gamma_{k_j}}
  \end{pmatrix}
  }
r^{-1/2}
=0,
\end{equation}
being $D_{k_j} \in \R^{2\times2}$ the invertible matrix
\begin{equation}\label{eq:defn.D} 
  D_{k_j}:= 
  \begin{cases}	
 \frac{1}{2\gamma(\lambda + k_j - \gamma_{k_j})}
  \begin{pmatrix}
    \lambda + k_j - \gamma_{k_j} & \nu-\mu \\
     -(\nu + \mu)        & -(\lambda + k_j - \gamma_{k_j})
  \end{pmatrix}
  \quad &\text{ if }\lambda + k_j - \gamma_{k_j} \neq 0,\\
      \frac{1}{-4\gamma_{k_j}^2}
    \begin{pmatrix}
      \mu - \nu            & 2\gamma_{k_j} \\
     2 \gamma_{k_j} & -(\nu + \mu) 
    \end{pmatrix}
    \quad &\text{ if }\lambda + k_j - \gamma_{k_j} = 0;
\end{cases}
\end{equation}
we set
\begin{equation}\label{eq:def.gamma.spherical.sotto}
\begin{pmatrix}
\Gamma_{m_j,k_j}^+(f_{m_j,k_j})\\
\Gamma_{m_j,k_j}^-(f_{m_j,k_j})
\end{pmatrix}:=
D_{k_j}\begin{pmatrix}
A^+\\
 A^-
 \end{pmatrix};
\end{equation}
\item \label{item:boundary.critico}
  if $ \delta_{k_j} = 0$, from \cite[Proposition 3.1, \emph{(iv)}]{coulombbiagio} we know that
  \begin{equation}
\label{eq:limit.f.gamma=0}
\begin{split}
&\lim_{r\to 0}
\abs*{
    \begin{pmatrix}
    f_{m_j,k_j}^+(r) \\ f_{m_j,k_j}^-(r)
  \end{pmatrix}
  - (M_{k_j} \log r+\mathbb{I}_2)
\begin{pmatrix}
A^+\\
A^-
\end{pmatrix}
 } 
 r^{-1/2}
 = 0, 
\end{split}
\end{equation}
being $M_{k_j}\in \R^{2\times2}$, $M_{k_j}^2=0$  defined as follows
\begin{equation}\label{eq:defn.M}
   M_{k_j}:= 
\begin{pmatrix}
  -(k_j+\lambda) & - \nu + \mu \\
  \nu + \mu           & k_j+\lambda
\end{pmatrix};
 \end{equation}
 we set
\begin{equation}\label{eq:def.gamma.spherical.critico}
\begin{pmatrix}
\Gamma_{m_j,k_j}^+(f_{m_j,k_j})\\
\Gamma_{m_j,k_j}^-(f_{m_j,k_j})
\end{pmatrix}:=
\begin{pmatrix}
A^+\\
A^-
\end{pmatrix};
\end{equation}
\item \label{item:boundary.sopracritico}
  if $ \delta_{k_j} < 0$, from \cite[Proposition 3.1, \emph{(v)}]{coulombbiagio}, we know that
\begin{equation}\label{eq:f0.gamma<0}
 \lim_{r\to 0}
\abs*{
    \begin{pmatrix}
    f_{m_j,k_j}^+(r) \\ f_{m_j,k_j}^-(r)
  \end{pmatrix}
  - E_{k_j} 
  \begin{pmatrix}
    A^+ r^{i\gamma_{k_j}} \\ A^- r^{-i\gamma_{k_j}}
  \end{pmatrix}
  }
r^{-1/2}
=0,
\end{equation}
being $E_{k_j} \in \C^{2\times2}$ the invertible matrix
\begin{equation}\label{eq:defn.E}
  E_{k_j}:= 
   \frac{1}{2i \gamma_{k_j}(\lambda + k - i\gamma_{k_j})}
  \begin{pmatrix}
    \lambda + k - i\gamma_{k_j} & \nu-\mu \\
     -(\nu + \mu)        & -(\lambda + k - i\gamma_{k_j})
  \end{pmatrix};
\end{equation}
we set
\begin{equation}\label{eq:def.gamma.spherical.sopra}
\begin{pmatrix}
\Gamma_{m_j,k_j}^+(f_{m_j,k_j})\\
\Gamma_{m_j,k_j}^-(f_{m_j,k_j})
\end{pmatrix}:=
E_{k_j}
\begin{pmatrix}
A^+\\
 A^-
 \end{pmatrix}.
\end{equation}
\end{enumerate}
Finally, set $\Gamma^+,\Gamma^-:\D(H_{max})\to \C^d$ as follows:
\begin{equation}\label{eq:def-Gamma}
\Gamma^{\pm}(\psi)= \left(\Gamma_{m_j,k_j}^{\pm}(f_{m_j,k_j})\right)_{(j,m_j,k_j)\in I}
 \in \C^d.
\end{equation}
Then, by definition, for any $(j,m_j,k_j)\in I$
\begin{equation}\label{eq:def-Gamma-abuse}
\left(\Gamma^{\pm}(\psi)\right)_{m_j,k_j}= \Gamma_{m_j,k_j}^{\pm}(f_{m_j,k_j})
 \in \C.
\end{equation}
\end{definition}

 We are now in position to state the main result of this paper.
\begin{theorem}[Boundary triples for  $H_{max}$]\label{thm:boundary}
  Let $H_{max}$ be defined as in \eqref{eq:defn.maximal.operator},
  let $d\in \mathbb{N}$ be as in \eqref{eq:defn.d} and 
  assume that $d>0$. Let $\Gamma^+,\Gamma^-$ be defined as in \eqref{eq:def-Gamma}.
Then, $(\C^d, \Gamma^+, \Gamma^-)$ is a boundary triple for
$H_{max}$.
\end{theorem}
\begin{remark}
  In general, boundary triples are not unique (see \cite[Proposition 1.14,
  Proposition 1.15]{bruning2008spectra}).  
  For example, a different boundary triple is determined 
  already by choosing an ordering of the triples different from the
  one in \eqref{eq:ordering}. 
\end{remark}
Thanks to the theory of the boundary triples, we can now describe all the self-adjoint extension of $\mathring H_{min}$: the following theorem is
consequence of \Cref{thm:boundary} and \Cref{thm:main-abs}:
\begin{theorem}\label{thm:main}
Let $\mathring H_{min}$ be defined as in \eqref{eq:defn.minimal.operator} and $d\in \mathbb{N}$ as in \eqref{eq:defn.d}.
The following hold:
\begin{enumerate}[label=({\roman*})]
\item if $d=0$, $\mathring H_{min}$ is essentially self-adjoint;
\item if $d>0$,  $\mathring H_{min}$ has many self-adjoint extensions. They can be classified in the following equivalent ways:
  \begin{itemize}
    \item For any $A,B \in \C^{d\times d}$, the extension $T_{A,B}$
with domain
\begin{equation}\label{eq:defn.TAB}
  \D(T_{A,B})=
 \seq{
 \psi\in \D(H_{max}):\ A \Gamma^+(\psi) =
 B \Gamma^-(\psi)}
\end{equation}
is self-adjoint if and only if
\begin{gather*}
  \label{eq:conditions.A,B.1}
  AB^* = B A^*,
  \\
  \label{eq:conditions.A,B.2}
  \text{ker}
  \begin{pmatrix}
    A & -B \\
    B & A
  \end{pmatrix}
  =0.
\end{gather*}
  \item There exists a one-to-one correspondence 
between the self-adoint extensions of ${\mathring H_{min}}$
and the unitary matrices $\mathcal{U}(d)$.
For $U\in \mathcal{U}(d)$, the corresponding self-adjoint extension $T_U$ has domain
 \begin{equation}\label{eq:descrizione.domini}
 \D(T_U)=
 \seq{
 \psi\in \D(H_{max}):\ i (\mathbb{I}_d + U  )\Gamma^+(\psi) =
 (\mathbb{I}_d - U) \Gamma^-(\psi)}.
\end{equation}
  \end{itemize}
  \end{enumerate}
\end{theorem}
\begin{remark}
  It is difficult to obtain the results of \Cref{thm:main} using
  Von Neumann's theory. Indeed, to exploit it, one has to
  find all the solutions to $(H_{max} \pm i) \psi = 0$, that is 
  hard to do for the general class of potentials considered in
  \eqref{eq:def.V}.
  By the way, Theorems 1.1, 1.2,
  1.3 in \cite{coulombbiagio} tell us that $h_{m_j,k_j}$ has deficiency indices
  $(1,1)$ if $\delta_{k_j} <1/4$ and $(0,0)$ if $\delta_{k_j} \geq 1/4$ on
  $C^\infty_c(0,+\infty)^2$.
  Consequently,
  $\mathring H_{min}$ has deficiency indices $(d,d)$, with $d$ defined
  as in \eqref{eq:defn.d}.
   We can now use the Von Neumann's theory, getting that 
   all the self-adjoint extensions of $\mathring H_{min}$
  are in one-to-one  correspondence with the unitary matrices
  $\mathcal{U}(d)$, but we can not provide an explicit bijection.
  Moreover, such correspondence does not describe the self-adjoint extensions: in \Cref{thm:main}  we provide a much clearer characterization of them in terms of the boundary behaviour in the origin of the functions in their domain.
\end{remark}

In the spirit of \cite{adv2013self,estebanloss,muller2016minimax} in the next theorem we select a \emph{distinguished}
self-adjoint extension among the ones defined in \Cref{thm:main}, requiring that
its domain is included in the domain of an appropriate quadratic form.
Let $q: C_c^{\infty}(\R^3;\C^4) \to \R$ be defined as
\begin{equation*}
  q(\psi):= \int_{\R^3} \Big[\abs{x}\abs{-i\alpha\cdot\nabla \psi}^2 
  -
  \abs{x}\abs{\V \psi}^2 
  \Big]\,dx.
\end{equation*}
If $\sup_{x\in \R^3} \abs{x}\abs{\V(x)} \leq 1$,
this form is symmetric and non-negative as a consequence of \eqref{eq:kato.nenciu},
and hence closable: we denote its closure $q$ (with abuse of notation)
and its maximal domain $\mathcal Q$.
In the following theorem, we consider $\V$ as in the class in \eqref{eq:def.V},
to exploit the complete description of all
the self-adjoint extensions  in \Cref{thm:main}.
We show that the condition $\mathcal{D}(T) \subset \mathcal{Q}$
selects a self-adjoint extension $T$ in the case that $\V$ is not a
critical anomalous magnetic potential, i.e.
$\V(x) \neq \pm i \alpha\cdot \hx \beta \abs{x}^{-1}$. Indeed, in this case
this approach does not select any extension, suggesting that it is not
possibile
to use this criterium for the general case.
\begin{theorem}\label{prop:critico}
Let $\mathring H_{min}$ be defined as in
\eqref{eq:defn.minimal.operator}, $\gamma_{k_j}$ as in \eqref{eq:defn.gammakj}, let
$d\in \mathbb{N}$ be defined as in \eqref{eq:defn.d} and assume that $d>0$.
Assume moreover that 
  \begin{equation}
    \label{eq:cond.V}
    \sup_{x\in \R^3} \abs{x}\abs{\V(x)} \leq 1,
    \quad
    \V(x)\neq \pm \frac{i \alpha\cdot \hx \beta}{\abs{x}}.    
  \end{equation}  
   Then there exists only one 
   self-adjoint extension $\mathring H_{min} \subseteq T_{A,B} \subseteq H_{max}$,
     such that $\mathcal D(T_{A,B}) \subseteq \mathcal Q$,
   with $A,B \in \C^{d \times d}$ determined by the following 
  conditions for all $\psi \in \D(H_{max})$: 
  \begin{enumerate}[label=({\roman*})]
  \item\label{item:case1} 
  and for all $(j,m_j,k_j)$ such that
    $0 \neq \gamma_{k_j} = k_j + \lambda$,
    \begin{equation}\label{eq:case1}
      (k_j + \lambda + \gamma_{k_j}) \left(\Gamma^+ (\psi)\right)_{m_j,k_j} =
      (\mu - \nu) \left( \Gamma^- (\psi)\right)_{m_j,k_j};
    \end{equation}
  \item\label{item:case2} for all $(j,m_j,k_j)$ such that
    $0 \neq \gamma_{k_j} \neq k_j + \lambda$:
    \begin{equation}\label{eq:case2}
      (\mu + \nu) \left(\Gamma^+ (\psi)\right)_{m_j,k_j}
      =
      -(k_j + \lambda - \gamma_{k_j}) 
      \left(\Gamma^- (\psi)\right)_{m_j,k_j};
    \end{equation}
    \item\label{item:case3} for all $(j,m_j,k_j)$ such that
      $\gamma_{k_j} =0 $,
      \begin{equation}\label{eq:case3a}
        (k_j + \lambda) \left(\Gamma^+ (\psi)\right)_{m_j,k_j} =
      (\mu - \nu) \left(\Gamma^- (\psi)\right)_{m_j,k_j},
    \end{equation}
    or equivalently
    \begin{equation}\label{eq:case3b}
      (\mu + \nu) \left(\Gamma^+ (\psi)\right)_{m_j,k_j}
      =
      -(k_j + \lambda ) \left(\Gamma^- (\psi)\right)_{m_j,k_j}.
    \end{equation}
  \end{enumerate}

\end{theorem}
\begin{remark}
  In the case that $\V$ is a general hermitian matrix-valued potential
  such that 
  $v:=\sup_{x\in \R^3} \abs{x}\abs{\V(x)} < 1$,
  a classification
  of all the self-adjoint extensions in the spirit of \Cref{thm:main}
  is not available. 
  However,
  it is still true that there exists only one self-adjoint extension
  whose domain in included in $\mathcal{Q}$.
  Indeed, thanks to \eqref{eq:kato.nenciu},
  for all $\psi \in C_c^{\infty}(\R^3)^4$
  \begin{equation}
    q(\psi) \geq 
    \int_{\R^3} \left[\abs{x}\abs{-i\alpha\cdot\nabla \psi}^2 
      -
      v^2 \frac{\abs{\psi}^2 }{\abs{x}}
    \right]\,dx 
    \geq  
    (1- v^2)   \int_{\R^3}
    \frac{\abs{\psi}^2 }{\abs{x}} \,dx,
  \end{equation}
  that immediately implies $\mathcal Q \subset \mathcal D(r^{-1/2})$. 
  If there exists a self-adjoint extension $T$ such that
  $\mathcal{D}(T) \subset \mathcal Q$, then it must be the
  distinguished one, the only one whose domain is contained in $\mathcal
  D(r^{-1/2})$, see \cite{klaus1979characterization}.
Vice-versa, constructing a self-adjoint extension with the property that $\mathcal D(T) \subseteq \mathcal Q$ is not trivial, and it is the subject of \cite{adv2013self}.
\end{remark}
\begin{remark}
  In the case that $\sup_{x\in \R^3} \abs{x}\abs{\V(x)} = 1$,
  the condition $\mathcal D(T) \subset \mathcal Q$
  appears not to be enough to select a self-adjoint
  extension $T$.
  Indeed, for 
  $\V(x)= \pm i \alpha\cdot \hx \beta/\abs{x}$,
  condition \eqref{eq:controllo.logaritmo}
  is true for all the functions in all the domains of self-adjointness.
   A similar phenomenon
was observed in  \cite[Remark 1.10]{coulombbiagio}.
\end{remark}
\begin{remark}\label{rmk:distinguished}
  As an application of \Cref{thm:main} and \Cref{prop:critico}, we
  describe the distinguished self-adjoint extension of the
  Dirac-Coulomb operator
  $H:=H_0 - \frac{\nu}{\abs{x}}\I_4$, for $\abs{\nu} \leq 1$:
  \begin{itemize}
  \item for $0\leq \abs{\nu}\leq \sqrt{3}/2$, $H$ is essentially
    self-adjoint;
  \item for $\sqrt{3}/2 < \abs{\nu} < 1$, we have that $d=4$,
    $\delta_1=\delta_{-1}= 1 - \nu^2 \in (0, 1/4)$,
    and $\Gamma^{\pm} = \left(\Gamma_{-\frac12,1}^{\pm},
    \Gamma_{\frac12,1}^{\pm},\Gamma_{-\frac12,-1}^{\pm},
    \Gamma_{\frac12,-1}^{\pm}\right)$.
    Then the distinguished extension has domain
    \begin{equation}\label{eq:defn.distinguished.explicit}
    \mathcal{D}(T_{A_\nu,\mathbb{I}_4}) =
    \{\psi \in \mathcal{D}(H_{max}): A_\nu\Gamma^+(\psi) =
    \Gamma^-  (\psi) \},
    \end{equation}
    with
    \begin{equation*}
      A_\nu: =
      \begin{pmatrix}
        \frac{\nu}{1 + \sqrt{1-\nu^2}} & 0 & 0 & 0 \\
        0 & \frac{\nu}{1 + \sqrt{1-\nu^2}} & 0 & 0 \\
        0 & 0 &  -\frac{\nu}{1 - \sqrt{1-\nu^2}} & 0 \\ 
        0 & 0 & 0 & -\frac{\nu}{1 - \sqrt{1-\nu^2}} 
      \end{pmatrix};
    \end{equation*}
  \item for $ \abs{\nu} = 1$, we have that $d=4$,
    $\delta_1=\delta_{-1}= 0$, $\Gamma^{\pm} = \left(\Gamma_{-\frac12,1}^{\pm},
    \Gamma_{\frac12,1}^{\pm},\Gamma_{-\frac12,-1}^{\pm},
    \Gamma_{\frac12,-1}^{\pm}\right)$, and
    the distinguished extension has domain
    $\mathcal{D}(T_{\nu\beta,\mathbb{I}_4})$.
  \end{itemize}
In the case that $\V=-1/\abs{x}$, \Cref{prop:critico} selects the
  distinguished self-ajoint extension, as defined in
  \cite{estebanloss}.
  More in general, in the case that $\V$ is as in \eqref{eq:def.V},
  \Cref{prop:critico} selects the distinguished extension, as in
  \cite[Propositions 1.7, 1.8]{coulombbiagio}.
\end{remark}
A fundamental tool in the proof of \Cref{prop:critico}
is the following improved version of \eqref{eq:kato.nenciu},
that we state independently.
\begin{lemma}\label{lem:hardy}
Let $\psi \in C_c^{\infty}(\R^3)^4$.
Then for all $R>0$
\begin{equation}\label{eq:hardy.optimal}
  \int_{\R^3} \abs{x}\abs*{-i\alpha\cdot\nabla \psi(x)}^2 \,dx
    \geq
  \int_{\R^3}  \frac{\abs{\psi(x)}^2}{\abs{x}}\,dx
  +
 \frac{1}{4}
  \int_{\R^3}
  \frac{\abs*{\psi(x) - \frac{R}{\abs{x}}\psi \big(R\frac{x}{\abs{x}}\big)}^2}{\abs{x}\log^2(\abs{x}/R)}\,dx.
\end{equation}
Moreover, the inequality is sharp.
\end{lemma}
\begin{remark}\label{lem:estebanloss}
  \Cref{lem:hardy} can be considered the analogous of
  \cite[Lemma 18]{els2017domains} in the general case
  \eqref{eq:cond.V}.
  Indeed, it allows to exclude a logaritmic decay
  in the origin
  for the functions in the domain of the self-adjoint extension.
 \end{remark}

The paper is organized as follows: in \Cref{sec:proof.boundary} we
prove \Cref{thm:boundary} and in 
 \Cref{sec:quadratic} we prove \Cref{lem:hardy} and \Cref{prop:critico}.

\section{Proof of \texorpdfstring{\Cref{thm:boundary}}{Theorem
    1.5}}\label{sec:proof.boundary}



We firstly prove the following lemma.
\begin{lemma}\label{thm:byparts.spherical.adjoint}
Let $j\in \seq{1/2,3/2,\dots}$, $m_j\in\seq{-j,\dots,j}$,
$k_j\in\seq{j+1/2,-j-1/2}$ such that $(j,m_j,k_j)\in I$
 and let $h^*_{m_j,k_j}$ be defined as in
 \eqref{eq:dirac.spherical*}. 
 Let $\Gamma_{m_j,k_j}^+, \Gamma_{m_j,k_j}^-$ be defined as in \Cref{def:boun-trip-sferical}.
 Then, $\left(\C,\Gamma_{m_j,k_j}^+, \Gamma_{m_j,k_j}^-\right)$ is a boundary triple for $h^*_{m_j,k_j}$.

\end{lemma}
\begin{proof}
  In this proof we will suppress the subscripts, since
  $j\in \seq{1/2,3/2,\dots}$, $m_j\in\seq{-j,\dots,j}$,
$k_j\in\seq{j+1/2,-j-1/2}$ are fixed.
  We distinguish various cases.

In the case  $0<\delta<\frac{1}{4}$,
thanks to \cite[Proposition 3.1, \emph{(iii)}]{coulombbiagio}, we have that
$f=(f^+,f^-) \in \mathcal{D}(h_{m_j,k_j}^*)$ if and only if
$f\in H^1(\epsilon,+\infty)^2$ for any $\epsilon>0$, and there exists $(A^+,A^-) \in \C^2$ 
such that \eqref{eq:limite.delta.12} holds true, for $D \in \R^{2\times2}$ defined in \eqref{eq:defn.D}.
Moreover, for any
$\widetilde{f}=(\widetilde{f}^+,\widetilde{f}^-)\in \mathcal{D}(h_{m_j,k_j}^*)$
we have
\begin{equation}\label{eq:det.gamma.<1/2}
\lim_{r\to 0} 
\begin{vmatrix}
f^+(r) & \overline{\widetilde{f}^+(r)}\\
f^-(r) & \overline{\widetilde{f}^-(r)}
\end{vmatrix}
= 
\begin{vmatrix}
D\begin{pmatrix}
A^+\\ A^-
\end{pmatrix}
\overline{
D\begin{pmatrix}
\widetilde A^+\\ \widetilde A^-
\end{pmatrix}
}
\end{vmatrix},
\end{equation}
where, with abuse of notation, we denoted
\begin{equation}\label{eq:abuse}
  \abs*{
    \begin{pmatrix}
      a \\ c
    \end{pmatrix}
    \begin{pmatrix}
      b \\ d
    \end{pmatrix}
  }
  :=
  \begin{vmatrix}
    a & b \\
    c & d
  \end{vmatrix}.
\end{equation}
Then for $f,\widetilde{f}\in \D(h_{m_j,k_j}^*)$, by the dominated convergence theorem, we have that 
\begin{equation}
\label{eq:is.it.symmetric}
\begin{split}
     &\int_0^{+\infty}  f \cdot
     \overline{h^*_{m_j,k_j}(\widetilde{f})}\,dr
     -
     \int_0^{+\infty} h^*_{m_j,k_j} (f) \cdot 
     \overline{\widetilde{f}}\,dr
    \\
    &= \lim_{\epsilon \to 0}
    \int_{\epsilon}^{+\infty} f  \cdot
    \overline{h^*_{m_j,k_j}(\widetilde{f})}\,dr
    -
    \int_{\epsilon}^{+\infty}  h^*_{m_j,k_j} (f)
    \cdot \overline{\widetilde{f}}\,dr
     = \lim_{\epsilon \to 0} 
    \begin{vmatrix}
      f^+(\epsilon) & \overline{\widetilde{f}^+(\epsilon)}\\
      f^-(\epsilon) & \overline{\widetilde{f}^-(\epsilon)}\\
    \end{vmatrix},
    \end{split}
\end{equation}
where in the last equality we used the fact that $f, \widetilde{f}\in H^1(\epsilon,+\infty)^2$.
We get \eqref{eq:boundarytriple.cond1}
combining in \eqref{eq:def.gamma.spherical.sotto},
\eqref{eq:det.gamma.<1/2} and \eqref{eq:is.it.symmetric}.
The surjectivity of the maps $\Gamma^+_{m_j,k_j}, \Gamma^-_{m_j,k_j}$
is easy to show: indeed let $(A^+,A^-)\in \C^2$ and let $f\in C^\infty(0,+\infty)^2$ such that
\[
f(r)= 
\begin{cases}
D
\begin{pmatrix}
A^+ r^\gamma\\
A^- r^{-\gamma}
\end{pmatrix}
 &\text{for}\ r<1,\\
0&\text{for}\ r>2.
\end{cases}
\]
Then $f\in \D(h^*_{m_j,k_j})$ and $\Gamma^\pm_{m_j,k_j}(f)$ are
defined as in \eqref{eq:def.gamma.spherical.sotto}.
Finally, \eqref{eq:boundarytriple.cond3} descends from the fact that $C^\infty_c(0,+\infty)^2\subset\operatorname{ker}(\Gamma^+_{m_j,k_j},\Gamma^-_{m_j,k_j})$.

Let us now consider the case that $\delta=0$. Thanks to
\cite[Proposition 3.1, \emph{(iv)}]{coulombbiagio}, 
$f=(f^+,f^-)\in \mathcal{D}(h_{m_j,k_j}^*)$ if and only if
$f\in H^1(\epsilon,+\infty)^2$ for any $\epsilon>0$, and there exists $(\Gamma_{m_j,k_j}^+(f), \Gamma_{m_j,k_j}^-(f)):=
(A^+, A^-) \in \C^2$ such that
\eqref{eq:limit.f.gamma=0} holds true, with $M\in \R^{2\times2}$,
$M^2=0$ defined as in \eqref{eq:defn.M}.
Moreover, for any
$\widetilde{f}=(\widetilde{f}^+,\widetilde{f}^-)\in \mathcal{D}(h_{m_j,k_j}^*)$
we have
\begin{equation}\label{eq:det.gamma.=0}
  \lim_{r\to 0}
    \begin{vmatrix}
     f^+(r) & \overline{\widetilde{f}^+(r)} \\
     f^-(r) & \overline{\widetilde{f}^-(r)}
  \end{vmatrix}
  =
  \begin{vmatrix}
    \Gamma^+(f) & \overline{ \Gamma^+(\widetilde f)} \\ 
    \Gamma^-(f) & \overline{\Gamma^-(\widetilde f)} 
  \end{vmatrix}.
\end{equation}
Reasoning as in the previous case, we get
\eqref{eq:boundarytriple.cond1}. Finally, \eqref{eq:boundarytriple.cond2} and \eqref{eq:boundarytriple.cond3} are proved as in
the previous case.

Let us lastly assume that $\delta<0$. In this case, thanks to \cite[Proposition 3.1, \emph{(v)}]{coulombbiagio} we have that
$f=(f^+,f^-)\in \mathcal{D}(h_{m_j,k_j}^*)$ if and only if  $f\in H^1(\epsilon,+\infty)^2$ for any $\epsilon>0$, and there exists $(A^+,A^-) \in \C^2$ such that
\eqref{eq:f0.gamma<0}  holds true, with $E \in \C^{2\times2}$ defined
as in \eqref{eq:defn.E}.
Moreover, for any
$\widetilde{f}=(\widetilde{f}^+,\widetilde{f}^-)\in \mathcal{D}(h^*_{m_j,k_j})$, with the same notation of \eqref{eq:abuse}, we get
\begin{equation}\label{eq:det.gamma.<0}
\lim_{r\to 0} 
\begin{vmatrix}
f^+(r) & \overline{\widetilde{f}^+(r)}\\
f^-(r) & \overline{\widetilde{f}^-(r)}
\end{vmatrix}
= 
\begin{vmatrix}
E\begin{pmatrix}
A^+\\ A^-
\end{pmatrix}
\overline{
E\begin{pmatrix}
\widetilde A^+\\ \widetilde A^-
\end{pmatrix}
}
\end{vmatrix},
\end{equation}
Due to \eqref{eq:def.gamma.spherical.sopra},
one get \eqref{eq:boundarytriple.cond1}, \eqref{eq:boundarytriple.cond2} and \eqref{eq:boundarytriple.cond3} reasoning as before.
\end{proof}

We are now ready to prove \Cref{thm:boundary}.
\begin{proof}[Proof of \Cref{thm:boundary}]
Let us start proving the condition \eqref{eq:boundarytriple.cond1}
in \Cref{defn:boundary}. Let
for any $\psi,\widetilde{\psi}\in\D(H_{max})$
such that
 \begin{equation}\label{eq:dec.H.max}
 H_{max}\psi=
 \sum_{j,m_j,k_j}
 h^*_{m_j,k_j}
 f_{m_j,k_j},
 \quad
 H_{max}\widetilde \psi=
 \sum_{j,m_j,k_j}
 h^*_{m_j,k_j}
\widetilde f_{m_j,k_j},
 \end{equation}
for appropriate $f_{m_j,k_j}$ and $\widetilde{f}_{m_j,k_j}$ in $\D(h^*_{m_j,k_j})$. Then
\begin{equation*}
\begin{split}
&\langle  \psi, H_{max}\widetilde{\psi}\rangle_{L^2(\R^3)^4}
-
\langle H_{max} \psi,\widetilde{\psi} \rangle_{L^2(\R^3)^4}
\\
&= 
\sum_{j,m_j,k_j}
\langle f_{m_j,k_j}, h^*_{m_j,k_j}\widetilde{f}_{m_j,k_j}\rangle_{L^2(0,\infty)^2}
-
\langle h^*_{m_j,k_j} f_{m_j,k_j},\widetilde{f}_{m_j,k_j}\rangle_{L^2(0,\infty)^2}\\
&=
\sum_{
\substack{j,m_j,k_j  \\ (k_j+\lambda)^2+\mu^2-\nu^2<1/4}
}
\langle f_{m_j,k_j}, h^*_{m_j,k_j}\widetilde{f}_{m_j,k_j}\rangle_{L^2(0,\infty)^2}
-
\langle h^*_{m_j,k_j} f_{m_j,k_j},\widetilde{f}_{m_j,k_j}\rangle_{L^2(0,\infty)^2},
\end{split}
\end{equation*}
where in the last equality we used the fact that $h^*_{m_j,k_j}$ is
self-adjoint when  $(k_j+\lambda)^2+\mu^2-\nu^2\geq 1/4$, as proved in \cite[Theorem 1.1]{coulombbiagio}.
Thanks to \Cref{thm:byparts.spherical.adjoint}, we conclude that 
\begin{equation}\label{eq:byparts.decomposition.spherical}
  \begin{split}
    \langle  \psi, H_{max}\widetilde{\psi}\rangle_{L^2(\R^3)^4}
    &-
    \langle H_{max} \psi,\widetilde{\psi} \rangle_{L^2(\R^3)^4}
    \\ &=
    \sum_{(j,m_j,k_j)\in I}
    \Gamma^+_{m_j,k_j} (f)\cdot\overline{\Gamma_{m_j,k_j}^- (\tilde{f})}
    -
    \Gamma^-_{m_j,k_j} (f)\cdot\overline{ \Gamma_{m_j,k_j}^+ (\tilde{f})},
  \end{split} 
\end{equation}
that gives immediately \eqref{eq:boundarytriple.cond1}. 


The surjectivity of $\Gamma^+$ and $\Gamma^-$ descends immediately from
the surjectivity of any $\Gamma^+_{m_j,k_j}$ and $\Gamma^-_{m_j,k_j}$
that has been showed in \Cref{thm:byparts.spherical.adjoint}. 

Finally, since $C_c^{\infty}(\R^3\setminus\{0\})^4 \subseteq \ker (\Gamma^+, \Gamma^-)$, we deduce the condition \eqref{eq:boundarytriple.cond2}.
\end{proof}
\section{Proof of \texorpdfstring{\Cref{prop:critico}}{Theorem 1.9}}
\label{sec:quadratic}
In this Section we prove \Cref{lem:hardy}, the following \Cref{prop:quadratic}
and finally \Cref{prop:critico}.
\begin{proof}[Proof of \Cref{lem:hardy}]
By direct computation  (see for example
   \cite[Equation (4.102)]{thaller})
  \begin{equation*}
    -i\alpha\cdot\nabla =  
    -i\alpha \cdot \hx
    \left(\partial_r + \frac{1}{\abs{x}} - \frac{1 + 2\SL}{\abs{x}}\right),
  \end{equation*}
  where $\S$ is the \emph{spin angular momentum operator} 
\begin{equation}
{\S}=
\frac{1}{2}
\begin{pmatrix}
\sigma & 0\\
0 & \sigma
\end{pmatrix}
.
\end{equation}
Consider  $\psi \in C_c^{\infty}(\R^3;\C^4)$.
Since $i\alpha\cdot \hx$ is a unitary matrix, we have
\begin{equation*}
  \begin{split}
    \int_{\R^3} \abs{x}\big\vert -i\alpha\cdot\nabla \psi \big\vert^2\,dx
    = &
    \int_{\R^3} \abs{x}
    \abs*{\Big(\partial_r + \frac{1}{\abs{x}} - \frac{1 + 2\SL}{\abs{x}}\Big)\psi}^2\,dx 
    \\
    = & 
    \int_{\R^3} \abs{x}
    \abs*{\Big(\partial_r + \frac{1}{\abs{x}}\Big)\psi}^2 \,dx
    +
    \int_{\R^3}
    \abs*{\frac{1 + 2\SL}{\abs{x}}\psi}^2 \,dx
    \\
    & 
    -2 \Re 
    \int_{\R^3} 
    \Big(\partial_r + \frac{1}{\abs{x}}\Big)\psi
\,
    \overline{
      (1 + 2\SL)\psi
    }
    \,dx.
  \end{split}
\end{equation*}
It is standard (see for example \cite[Lemma 2.1]{coulombluis}) to show that the last term in the previous equation vanishes, indeed 
$1 + 2\SL$ and $\partial_r + \frac{1}{\abs{x}}$ are respectively symmetric and skew-symmetric 
on $C_c^{\infty}(\R^3)^4$, and the two operators commute with each other.

Let $\phi:= \abs{x}\psi$. We have that 
$\partial_r \phi = \abs{x} (\partial_r + \abs{x}^{-1}) \psi$ and consequently
\begin{equation*}
    \int_{\R^3} \abs{x}
    \abs*{\Big(\partial_r + \frac{1}{\abs{x}}\Big)\psi}^2 \,dx
    =
    \int_0^{+\infty} \int_{\mathbb S^2} r \abs{\partial_r \phi(r\omega)}^2 \,d\omega dr
\end{equation*}
Thanks to Proposition 2.4, \emph{(iii)} in \cite{coulombbiagio}, 
\begin{equation*}
       \int_{\mathbb S^2} \int_0^{+\infty} r \abs{\partial_r \phi(r\omega)}^2 \, dr d\omega
      \geq 
      \frac14
       \int_{\mathbb S^2}
       \int_0^{+\infty}
       \frac{\abs{\phi(r\omega) - \phi(R \omega)}^2}{r \log^2(r/R)} \,dr d\omega.
\end{equation*}
This inequality is sharp, as underlined in \cite[Remark 2.5]{coulombbiagio}.
Observing that $\abs{1+2\SL} \geq 1$, we finally get the thesis. 
\end{proof}

\begin{proposition}\label{prop:quadratic}
  For all $\psi \in \mathcal{Q}$
  \begin{equation}\label{eq:controllo.logaritmo}
  \int_{\{\abs{x}<1\}} \frac{\abs{\psi(x)}^2}{\abs{x}\log^2{\abs{x}}}\, dx < +\infty.
\end{equation}
\end{proposition}
\begin{proof}
We show that for all $\psi \in \mathcal Q$
\begin{equation}\label{eq:prop.q}
  q(\psi) \geq
   \frac{1}{4}
  \int_{\R^3}
  \frac{\abs*{\psi(x) - \frac{R}{\abs{x}}\psi \big(R\frac{x}{\abs{x}}\big)}^2}{\abs{x}\log^2(\abs{x}/R)}\,dx.
\end{equation}
Since 
$Q=\overline{C^\infty_c(\R^3)}^{\norm{\cdot}_{q}}$,
with $\norm{\cdot}_q^2 := q(\cdot) + \norm{\cdot}_2^2$, there
exists a sequence $(\psi_j)_j
\subset C_c^{\infty}(\R^3)$ such that
$\norm{\psi-\psi_j}_q \to 0$ and $\psi - \psi_j \to 0$ 
almost everywhere as $j \to +\infty$.
Since \eqref{eq:hardy.optimal} holds for $\psi_j - \psi_m \in
C_c^{\infty}(\R^3)$,  $(\chi_j)_j$ is a Cauchy sequence
in $L^2(\R^3,\abs{x}^{-1}dx)$, for
\begin{equation*}
  \chi_j(x) := \frac{\psi_j(x) - \psi_j (R {x}/{\abs{x}})}%
  {\log(\abs{x}/R)}.
\end{equation*}
Consequently, $\chi_j \to \chi \in L^2(\R^3,\abs{x}^{-1}dx)$.
On the other hand, since $\psi_j \to \psi$  almost
everywhere, then $\chi_j \to \frac{\psi - \psi\big(R
  \frac{\cdot}{\abs{\cdot}}\big)}{\log(\abs{x}/R)}$ almost everywhere,
and we conclude that $\chi_j \to \frac{\psi - \psi\big(R
  \frac{\cdot}{\abs{\cdot}}\big)}{\log(\abs{x}/R)}$ in
$L^2(\R^3,\abs{x}^{-1}dx)$. In conclusion, \eqref{eq:prop.q}
holds for $\psi \in \mathcal Q$.

Consequently,
  \begin{equation*}
   \int_{\{\abs{x}<1\}}\!\!
   \frac{\abs{\psi(x)}^2}{\abs{x}\log^2(\abs{x}/R)}\,dx 
   \leq
   2 \int_{\{\abs{x}<1\}}\!\!
   \frac{\abs*{\psi(x)-  \frac{R}{\abs{x}}
       \psi\big(R \frac{x}{\abs{x}}\big)}^2}%
   {\abs{x}\log^2(\abs{x}/R)}\,dx 
   + 2
     \int_0^1
     \frac{\frac{R^2}{r^2}
       \int_{\{\abs{x}=r\}}\!\!
       \abs*{\psi\big(R \frac{x}{|x|}\big)}^2 \,dS_x}%
   {r \log^2(r/R)}\,dr.
 \end{equation*}
The second term at right hand side is finite, since the numerator is
constant with respect to $r \in (0,1)$ and $(r \log^2{r})^{-1}$ is
integrable in the origin, and the first term at right
hand side is finite, as it is shown above.
\end{proof}

We can now finally prove \Cref{prop:critico}.
\begin{proof}[Proof of \Cref{prop:critico}]
  We firstly show that
  $\gamma_{k_j} \geq 0$ for all $j=1/2,3/2,\dots$,
  that is 
  $(k+\lambda)^2 + \mu^2 - \nu^2 \geq 0$
  for all $k \in \Z \setminus \{0\}$.
  Indeed, since $\abs{x}\abs{\V(x)} = \abs{\nu} + \sqrt{\mu^2 +
    \lambda^2} \leq 1$, then $\nu^2 \leq 1 + \mu^2 + \lambda ^2 - 2
  \sqrt{\mu^2 + \lambda^2}$.
  Moreover, since $\abs{\lambda} \leq 1$, then
  $M:=\min_{k \in \Z \setminus \{0\}} (k+\lambda)^2 + \mu^2 - \nu^2
  = (1-\abs{\lambda})^2 + \mu^2 - \nu^2$.
  Assume by contradiction that $M<0$. Then
  $(1-\abs{\lambda})^2 + \mu^2 < \nu^2 \leq 
  1 + \mu^2 + \lambda ^2 - 2
  \sqrt{\mu^2 + \lambda^2}$,
  that is $ \abs{\lambda} > \sqrt{\mu^2 + \lambda^2}$ and this is
  absurd.
  Incidentally we remark that $M = 0$ only if $\mu=0$.

We denote 
\begin{align*}
I_1&:=\seq{(j,m_j,k_j)\in I:0 \neq \gamma_{k_j} = k_j + \lambda},\\
I_2&:=\seq{(j,m_j,k_j)\in I: 0 \neq \gamma_{k_j} \neq k_j + \lambda},\\
I_3&:=\seq{(j,m_j,k_j)\in I:\gamma_{k_j} =0}.
\end{align*}

Following \eqref{eq:ordering},  we identify
  \begin{equation*}
  s \in \seq{1,\dots,d} \leftrightarrow (j,m_j,k_j) \in I.
  \end{equation*}

  Thanks to this, we have that $\{I_1,I_2,I_3\}$ is a partition of $\{1,\dots,d\}$.

  In the following we determine $A,B \in
  \C^{d\times d}$ in such a way that $\mathcal{D}(T_{A,B}) \subseteq
  \mathcal{Q}$.
  Let $\psi$ be a generic element in $\mathcal{D}(T_{A,B})$.
  Thanks to \Cref{prop:quadratic}, the condition 
  $\mathcal{D}(T_{A,B}) \subseteq \mathcal{Q}$ implies that $\psi$ verifies
  \eqref{eq:controllo.logaritmo}.
  Following the notations  of \Cref{thm:boundary},
  we denote
  \begin{equation*}
    \begin{split}
      \psi(x) =&
      \sum_{j=\frac{1}{2},\frac{3}{2},\dots}^\infty
      \,   \sum_{m_j=-j}^j \,
      \sum_{k_j=\pm(j+1/2)} \,
      \frac{1}{r}\left(
        f^+_{m_j,k_j}(r)\Phi^+_{m_j,k_j}(\hx)+
        f^-_{m_j,k_j}(r)\Phi^-_{m_j,k_j}(\hx)\right),
      \\
      f_{m_j,k_j} =& \, (f^+_{m_j,k_j}, f^-_{m_j,k_j}).
    \end{split}
  \end{equation*}
  For all $(j,m_j,k_j) \in I_1\cap I_2$, we have that $f_{m_j,k_j}$ verifies \eqref{eq:limite.delta.12}:
  since the singular behaviour is not allowed by
  \eqref{eq:controllo.logaritmo},
  we have necessarily that $A^- = 0$.
  Thanks to \eqref{eq:def.gamma.spherical.sotto},
  we have that this is equivalent to \eqref{eq:case1} when $(j,m_j,k_j)\in I_1$
  and equivalent to \eqref{eq:case2} when $(j,m_j,k_j)\in I_2$.
  We define the matrices $A$ and $B$ accordingly:
  \[
  \begin{array}{lll} 
    A_{ss}:= k_j + \lambda + \gamma_{k_j}, &   B_{ss}:= \mu - \nu, 
    &\text{ for }s\sim (j,m_j,k_j)\in I_1 \\
    A_{ss}:= \mu + \nu,    & B_{ss}:=  -(k_j + \lambda - \gamma_{k_j}), &
    \text{ for } s\sim (j,m_j,k_j)\in I_2, \\
    A_{st}=B_{st}=0, & &\text{ for }s \sim (j,m_j,k_j)\in I_1\cup I_2, \,
                         1\leq t\leq d, t\neq s.
  \end{array}
\]
  For all $(j,m_j,k_j) \in I_3$, we have that $f_{m_j,k_j}$ verifies \eqref{eq:limit.f.gamma=0}:
  since the logarithmic behaviour is not allowed by
  \eqref{eq:controllo.logaritmo},
  we have necessarily that $\text{Ran}(\Gamma_{m_j,k_j}^+,\Gamma_{m_j,k_j}^-)\subseteq \ker M_{k_j}$.
  This gives \eqref{eq:case3a} and \eqref{eq:case3b}: they are equivalent since
  $M$ has rank 1.
  Using the identification $s\sim (j,m_j,k_j)$,
  we define $A$ and $B$ accordingly:
  \begin{equation}
    A_{ss}:= k_j + \lambda, \quad B_{ss}:=\mu - \nu,
    \quad \text{ for }s s\sim (j,m_j,k_j)\in I_3, 
  \end{equation}
  or equivalently
  \begin{equation}
    A_{ss}:= \mu + \nu, \quad B_{ss}:= -(k_j + \lambda ),
    \quad \text{ for }s\sim (j,m_j,k_j)\in I_3.
  \end{equation}
  and $A_{st}=B_{st}=0$, for $s \sim (j,m_j,k_j)\in I_3$, 
  $t\in\{1,\dots,d\}, t\neq s$.
  
  In order to show that the extension that we have built is
  self-adjoint, we check the conditions
  \eqref{eq:conditions.A,B.1-abs} and \eqref{eq:conditions.A,B.1-abs} in
  \Cref{thm:main-abs}:
  since $A$ and $B$ are real and diagonal
  we have that
  \begin{equation*}
    AB^* = AB = BA = BA^*,
  \end{equation*}
  that is \eqref{eq:conditions.A,B.1-abs}.
  In order to show that \eqref{eq:conditions.A,B.2-abs}, we
  show equivalently that $\det(AA^* + BB^*)\neq 0$ (see
  \cite[Section 125, Theorem 4]{akhiezer2013theory}).
  Indeed, the matrix $AA^* + BB^*$ is diagonal and the elements of
  the diagonal equal $C_{ss}:=(A_{ss})^2 + (B_{ss})^2$, for $s=1,\dots,d$.
  For $s \in I_1$ we have that
  $C_{ss} = (k_j + \lambda + \gamma_{k_j})^2 + (\mu - \nu)^2 \geq
  (k_j + \lambda + \gamma_{k_j})^2 = 4\gamma_{k_j}^2 > 0$.
  For $s \in I_2$ we have that
  $C_{ss}=(\mu + \nu)^2 +  (k_j + \lambda + \gamma_{k_j})^2
  \geq (k_j + \lambda - \gamma_{k_j})^2 > 0$.
  Finally, for $s \in I_3$, we have that
  $C_{ss} = (k_j + \lambda)^2 + (\mu - \nu)^2$ or
  $C_{ss} = (k_j + \lambda)^2 + (\nu + \mu)^2$:
  in both cases, $C_{ss} = 0$ if and only if 
  $(\nu,\mu,\lambda)=(0,0,1)$ or
  $(\nu,\mu,\lambda)=(0,0,-1)$, but this is excluded by \eqref{eq:cond.V}.

  The linear relation associated to $A,B$ determines uniquely a
  unitary matrix $U \in \mathcal{U}(d)$ such that $T_{A,B}=T_U$,
  defined as in \eqref{eq:descrizione.domini}, see
  \cite[Section 2]{PANKRASHKIN2006207}, \cite[Theorem
  4.6]{arens1961operational}, \cite[Theorem
  3.1.4]{gorbachuk2012boundary}.
  This implies that  $T_{A,B}$ is the unique self-adjoint extension with
  the required properties and concludes the proof.
\end{proof}



\section*{Acknowledgements}
This work was mainly developed while both authors were employed at \emph{BCAM - Basque Center for Applied Mathematics}, 
and they were supported by
ERCEA Advanced Grant 2014 669689 - HADE, by the MINECO project MTM2014-53850-P, by
Basque Government project IT-641-13 and also by the Basque Government through the BERC
2018-2021 program and by Spanish Ministry of Economy and Competitiveness MINECO: BCAM
Severo Ochoa excellence accreditation SEV-2017-0718. 
The first author
also acknowledges the Istituto Italiano di Alta Matematica ``F.~Severi'' and
 the Czech Science Foundation 
(GA\v{C}R) within the project 17-01706S.
The second author also has received funding from
the European Research Council (ERC) under the European Union’s Horizon 2020
research and innovation programme (grant agreement MDFT No 725528 of M.L.).

\end{document}